\documentclass{article}

\usepackage{float}
\floatstyle{ruled}
\newfloat{model}{thp}{lop}
\floatname{model}{Model}

\usepackage{multicol}

\usepackage{xcolor}

\newcommand{\refacepf}{(\hyperref[model:ac_pf_e]{AC-E-PF})}
\newcommand{\refsocepf}{(\hyperref[model:soc_pf_e]{SOC-E-PF})}

\newcommand{\refnfepf}{(\hyperref[model:nf_pf_e]{NF-E-PF})}
\newcommand{\refcpepf}{(\hyperref[model:cp_pf_e]{CP-E-PF})}

\usepackage{xspace}

\usepackage{url}
\usepackage{amssymb,amsmath,amsthm,bm,mathtools,mathabx}
\usepackage{graphicx}
\usepackage{authblk}

\usepackage{hyperref} 
\hypersetup{
colorlinks=true, 
breaklinks=true,
urlcolor= blue, 
linkcolor= blue,
citecolor=red, 
}

\newtheorem{theorem}{Theorem}[section]
\newtheorem{lemma}[theorem]{Lemma}

\newtheorem{corollary}[theorem]{Corollary}

\begin{document}

\title{Network Flow and Copper Plate Relaxations \\ for AC Transmission Systems} 
\author[1,2]{Carleton Coffrin} 
\author[1,2]{Hassan L. Hijazi}
\author[1,2]{Pascal Van Hentenryck}

\affil[1]{Optimisation Research Group, NICTA}
\affil[2]{College of Engineering and Computer Science, Australian National University}
\maketitle

\abstract{
Nonlinear convex relaxations of the power flow equations and, in
particular, the Semi-Definite Programming (SDP), Convex Quadratic (QC), and Second-Order
Cone (SOC) relaxations, have attracted significant interest in recent years. 
Thus far, little attention has been given to simpler linear relaxations of the power flow equations, which may bring significant performance gains at the cost of model accuracy.
To fill the gap, this paper develops two intuitive linear relaxations of the power flow equations, one based on classic network flow models (NF) and another inspired by copper plate approximations (CP).
Theoretical results show that the proposed NF model is a relaxation of the established nonlinear SOC model and the CP model is a relaxation of the NF model.
Consequently, considering the linear NF and CP relaxations alongside the established nonlinear relaxations (SDP, QC, SOC) provides a rich variety of tradeoffs between the relaxation accuracy and performance.
}

\section*{Nomenclature}
\begin{multicols}{2} 

\begin{description}
  \item [{$N$}]  - The set of nodes in the network 
  \item [{$E$}]  - The set of {\em from} edges in the network 
  \item [{$E^R$}]  - The set of {\em to} edges in the network 
  %
  \item [{$\bm i$}] - imaginary number constant
  \item [{$I$}] - AC current
  \item [{$S = p+ \bm iq$}] - AC power
  \item [{$V = v \angle \theta$}]  - AC voltage
  \item [{$Z = r+ \bm ix$}] - Line impedance
  \item [{$Y = g + \bm ib$}]  - Line admittance
  \item [{$T = t \angle \theta^t$}]  - Transformer properties
  \item [{$Y^s = g^s + \bm ib^s$}]  - Bus shunt admittance
  \item [{$W $}]  - Product of two AC voltages
  \item [{$L$}]  - Current magnitude squared, $|I|^2$
  %
  \item [{$b^c$}] - Line charging
  \item [{$s^u$}] - Line apparent power thermal limit
  \item [{$\theta^\Delta$}] - Phase angle difference limit
  \item [{$S^d = p^d+ \bm iq^d$}] - AC power demand
  \item [{$S^g = p^g+ \bm iq^g$}] - AC power generation
  \item [{$c_0,c_1,c_2$}] - Generation cost coefficients 
 %
   \item [{$\Re(\cdot)$}] - Real component of a complex number
   \item [{$\Im(\cdot)$}] - Imaginary component of a complex number
   \item [{$(\cdot)^*$}] - Conjugate of a complex number
   \item [{$|\cdot|$}] - Magnitude of a complex number, $l^2$-norm
   \item [{$\angle$}] - Angle of a complex number
  %
  %
  \item [{$x^u$}] - Upper bound of $x$
  \item [{$x^l$}] - Lower bound of $x$
  \item [{$\widecheck{x}$}] - Convex envelope of $x$
  \item [{$\bm x$}] - A constant value
\end{description}

\end{multicols}

\clearpage
\section{Introduction}
\label{sec:intro}

Nonlinear convex relaxations of the power flow equations have attracted
significant interest in recent years. They include the Semi-Definite
Programming (SDP) \cite{Bai2008383}, Second-Order Cone (SOC)
\cite{Jabr06}, Convex-DistFlow (CDF) \cite{6102366}, and the recent
Quadratic Convex (QC) \cite{QCarchive} and Moment-Based \cite{7038397,6980142} relaxations.
Much of the excitement underlying this line of research comes from the fact that
the SDP relaxation has shown to be tight on a variety of case studies
\cite{5971792}, opening a new avenue for accurate, reliable, and
efficient solutions to a variety of power system applications. Indeed,
industrial-strength optimization tools (e.g., Gurobi \cite{gurobi},
cplex \cite{cplex}, Mosek \cite{mosek}) are now available to solve
various classes of convex optimization problems. 

Thus far, study of power flow relaxations has focused primarily on these nonlinear methods in the interest of the highest possible model accuracy at any cost.
Subsequently, little attention has been given to simpler linear relaxations \cite{5772044,linear_opf}, which are less accurate but have significant performance and scalability benefits.
To fill that gap, this paper develops two intuitive linear relaxations of the power flow equations, one based on classic network flow models \cite{ahuja1993network} and another inspired by copper plate approximations.
The main contributions of this work can be summarized as:
\begin{enumerate}
\item Developing two intuitive linear relaxations for AC transmission system optimization, the network flow (NF) and copper plate (CP) relaxations.
\item Proving that NF is a linear relaxation of the established nonlinear SOC relaxation and that CP is a relaxation of NF.
%
%
\end{enumerate}
It is important to emphasize that, contrary to the models
  presented herein, the established DC power flow model \cite{Stott_2009bb} is a linear {\em approximation} of
  the AC power flow, not a relaxation. As a consequence, it cannot be
  used for providing quality guarantees (i.e. lower bounds) on the underlying AC power flow model.  
  
The rest of
the paper is organized as follows.  Section \ref{sec:ac_ex} reviews
the extended formulation of non-convex AC power flow feasibility
problem for transmission systems (AC-E-PF).  Section \ref{sec:relax_e}
develops the relaxations of (AC-E-PF), it begins with a review of the
SOC relaxation and then proposes the NF and CP relaxations.  Section
\ref{sec:lin_relax} reviews the related works on linear relaxations
and illustrates how all of the relaxations considered perform on a
well known 3-bus example.  Section \ref{sec:conc} concludes the paper.

\section{Extended AC Power Flow for Transmission Systems}
\label{sec:ac_ex}

In the interest of clarity, AC Power Flows, and their relaxations, are most often presented on the simplest version of the AC power flow equations.  However, transmission system test cases include additional parameters such as bus shunts ($\bm {Y^s}$), line charging ($\bm {b^c}$), and transformers ($\bm T$), which complicate the AC power flow equations significantly.
This section presents an extended AC power flow problem including all of these additional parameters.

A power network is composed of a variety of components such as buses, lines, generators, and loads.  
The network can be interpreted as a graph $(N,E)$ where the set of buses $N$ represent the nodes and the set of lines $E$ represent the edges. 
Note that $E$ is an undirected set of edges, however each edge $(i,j) \in E$ is assigned a {\em from} side $(i,j)$ and a {\em to} side $(j,i)$, arbitrarily.
These two sides are critically important as power is lost as it flows from one side to another. 
Lastly, to break numerical symmetries in the model and to allow easy comparison of solutions, a reference node $r \in N$ is also specified.

The AC power flow equations are based on complex quantities for 
current $I$, voltage $V$, admittance $Y$, transformers $T$, and power $S$, 
which are linked by the physical properties of
Kirchhoff's Current Law (KCL),
\begin{align}
& I^g_i - {\bm I^d_i} - {\bm Y^s_i} V_i = \sum_{(i,j)\in E} I_{ij} + \sum_{(j,i)\in E} I_{ji} \;\; \forall i \in N \label{eq:kcl_e}
\end{align}
Ohm's Law, i.e.,
\begin{align}
& I_{ij} = \left( \bm Y_{ij} + \bm i \frac{\bm {b^c}_{ij}}{2} \right) \frac{V_i}{\bm T_{ij}^*} - \bm Y_{ij} V_j  \;\; \forall (i,j),(j,i) \in E \label{eq:ohm_e}
\end{align}
and the definition of AC power, i.e.,
\begin{align}
& S_{ij} = \frac{V_i}{\bm T_{ij}^*} I_{ij}^*  \;\; \forall (i,j)\in E \label{eq:power_e}
\end{align}
Combining these three properties yields the extended AC Power Flow equations,
\begin{subequations}
\begin{align}
& S^g_i - {\bm S^d_i} - (\bm Y^s_i)^* |V_i|^2 = \sum_{(i,j)\in E} S_{ij} + \sum_{(j,i)\in E} S_{ji} \;\; \forall i\in N \label{eq:bus_s_e} \\ 
& S_{ij} = \left( \bm Y^*_{ij} - \bm i\frac{\bm {b^c}_{ij}}{2} \right) \frac{|V_i|^2}{|\bm{T}_{ij}|^2} - \bm Y^*_{ij} \frac{V_iV^*_j}{\bm{T}^*_{ij}} \;\; (i,j)\in E \label{eq:line_s_e_from}\\
& S_{ji} = \left( \bm Y^*_{ij} - \bm i\frac{\bm {b^c}_{ij}}{2} \right) |V_j|^2 - \bm Y^*_{ij} \frac{V^*_iV_j}{\bm{T}_{ij}} \;\; (i,j)\in E \label{eq:line_s_e_to}
\end{align}
\end{subequations}
Observe that $\sum$ over $(i,j)\in E$ collects the edges oriented in the 
{\em from} direction and $\sum$ over $(j,i)\in E$ collects the edges oriented in the {\em to} direction around bus $i\in N$.
These non-convex nonlinear equations define how power flows in the
network and are a core building block in many power system
applications. Additionally, many applications require various
operational side constraints on the flow of power. We now review some 
of the most significant ones.

\paragraph{Generator Capabilities}

AC generators have limitations on the amount of active and reactive
power they can produce $S^g$, which is characterized by a generation
capability curve \cite{9780070359581}.  Such curves typically define
nonlinear convex regions which are often approximated by boxes in
AC transmission system test cases, i.e.,
\begin{align}
& \bm {S^{gl}}_i \leq S^g_i \leq \bm {S^{gu}}_i \;\; \forall i \in N 
\end{align}

\paragraph{Line Thermal Limit}

AC power lines have thermal limits \cite{9780070359581} to prevent
lines from sagging and automatic protection devices from activating.
These limits are typically given in Volt Amp units and constrain 
the apparent power flows on the lines, i.e.,
\begin{align}
& |S_{ij}| \leq \bm {s^u}_{ij} \;\; \forall (i,j),(j,i) \in E
\end{align}

\paragraph{Bus Voltage Limits}

Voltages in AC power systems should not vary too far (typically $\pm
10\%$) from some nominal base value \cite{9780070359581}.  This is
accomplished by putting bounds on the voltage magnitudes, i.e.,
\begin{align}
& \bm {v^l}_i \leq |V_i| \leq \bm {v^u}_i \;\; \forall i \in N
\end{align}
%

\paragraph{Phase Angle Differences}

Small phase angle differences are also a design imperative in AC power
systems \cite{9780070359581} and it has been suggested that phase
angle differences are typically less than $10$ degrees in practice
\cite{Purchala:2005gt}. These constraints have not typically been
incorporated in AC transmission test cases \cite{matpower}. However,
recent work \cite{LPAC_ijoc,QCarchive} have observed that
incorporating Phase Angle Difference (PAD) constraints, i.e.,
\begin{align}
&  -\bm {\theta^\Delta}_{ij} \leq \angle \! \left( V_i V^*_j \right) \leq \bm {\theta^\Delta}_{ij} \;\; \forall (i,j) \in E \label{eq:pad_1}
\end{align}
is useful in the convexification of the AC power flow equations. For
simplicity, this paper assumes that the phase angle difference bounds
are symmetrical and within the range $(- \bm \pi/2, \bm \pi/2 )$, i.e.,
\begin{align}
& 0 \leq \bm {\theta^{\Delta}}_{ij} \leq \frac{\bm \pi}{2} \;\; (i,j) \in E \nonumber
\end{align}
but the results presented here can be extended to more general
cases. Observe also that the PAD constraints \eqref{eq:pad_1} can be
implemented as a linear relation of the real and imaginary components
of $V_iV^*_j$ \cite{6822653}, i.e.,
\begin{align}
& \tan(-\bm {\theta^\Delta}_{ij}) \Re\left(V_iV^*_j\right) \leq \Im\left(V_iV^*_j\right) \leq \tan(\bm {\theta^\Delta}_{ij}) \Re\left(V_iV^*_j\right) \;\; \forall (i,j) \in E \label{eq:w_pad}
\end{align}
and that equation \eqref{eq:line_s_e_from} can be used to express these in terms of the $S$ variables as follows,
\begin{align}
& V_iV^*_j = \bm Z^*_{ij}  \bm{T}^*_{ij} \left( \left( \bm Y^*_{ij} - \bm i\frac{\bm {b^c}_{ij}}{2} \right) \frac{|V_i|^2}{|\bm T_{ij}|^2} - S_{ij} \right) \;\; (i,j) \in E \label{eq:vv_factor_e}
\end{align}
These equations combined with \eqref{eq:w_pad} implement the PAD constraints in terms of the $V$ and $S$ variables as follows,  
\begin{align}
 \tan(-\bm {\theta^\Delta}_{ij}) \Re\left( \bm Z^*_{ij}  \bm{T}^*_{ij} \left( \left( \bm Y^*_{ij} - \bm i\frac{\bm {b^c}_{ij}}{2} \right) \frac{|V_i|^2}{|\bm T_{ij}|^2} - S_{ij} \right) \right) &\leq \Im\left( \bm Z^*_{ij}  \bm{T}^*_{ij} \left( \left( \bm Y^*_{ij} - \bm i\frac{\bm {b^c}_{ij}}{2} \right) \frac{|V_i|^2}{|\bm T_{ij}|^2} - S_{ij} \right) \right) \;\; \forall (i,j) \in E \label{eq:s_pad_1} \\
 \tan(\bm {\theta^\Delta}_{ij}) \Re\left( \bm Z^*_{ij}  \bm{T}^*_{ij} \left( \left( \bm Y^*_{ij} - \bm i\frac{\bm {b^c}_{ij}}{2} \right) \frac{|V_i|^2}{|\bm T_{ij}|^2} - S_{ij} \right) \right) &\geq \Im\left( \bm Z^*_{ij}  \bm{T}^*_{ij} \left( \left( \bm Y^*_{ij} - \bm i\frac{\bm {b^c}_{ij}}{2} \right) \frac{|V_i|^2}{|\bm T_{ij}|^2} - S_{ij} \right) \right)  \;\; \forall (i,j) \in E \label{eq:s_pad_2} 
\end{align}
The usefulness of these alternate formulations will be apparent later in the paper.

\paragraph{Other Constraints}
Other line flow constraints have been proposed, such as, active power
limits and voltage difference limits \cite{5971792,6822653}.  However,
we do not consider them here since, to the best of our knowledge, test
cases incorporating these constraints are not readily available.

\begin{model}[t]
\caption{The Extended AC Power Flow Feasibility Problem (AC-E-PF)}
\label{model:ac_pf_e}
\begin{subequations}
\begin{align}
\mbox{\bf variables:} \nonumber \\
& S^g_i \in ( \bm {S^{gl}}_i, \bm {S^{gu}}_i) \;\; \forall i\in N \nonumber \\
& V_i \in ( \bm {V^l}_i, \bm {V^u}_i ) \;\; \forall i\in N \nonumber \\
& S_{ij} \in (\bm {S^{l}}_{ij},\bm {S^{u}}_{ij})\;\; \forall (i,j),(j,i) \in E \nonumber \\
%
\mbox{\bf subject to:} \nonumber \\
& \angle V_{\bm r} = 0 \label{eq:ac_e_0} \\
& \bm {v^l}_i \leq |V_i| \leq \bm {v^u}_i \;\; \forall i \in N \label{eq:ac_e_1} \\
& S^g_i - {\bm S^d_i} - (\bm Y^s_i)^* |V_i|^2 = \sum_{\substack{(i,j)\in E}} S_{ij} + \sum_{\substack{(j,i)\in E}} S_{ji} \;\; \forall i\in N \label{eq:ac_e_3} \\ 
& S_{ij} = \left( \bm Y^*_{ij} - \bm i\frac{\bm {b^c}_{ij}}{2} \right) \frac{|V_i|^2}{|\bm{T}_{ij}|^2} - \bm Y^*_{ij} \frac{V_iV^*_j}{\bm{T}^*_{ij}} \;\; (i,j)\in E \label{eq:ac_e_4} \\
& S_{ji} = \left( \bm Y^*_{ij} - \bm i\frac{\bm {b^c}_{ij}}{2} \right) |V_j|^2 - \bm Y^*_{ij} \frac{V^*_iV_j}{\bm{T}_{ij}} \;\; (i,j)\in E   \label{eq:ac_e_5} \\
& |S_{ij}| \leq \bm {s^u}_{ij} \;\; \forall (i,j),(j,i) \in E \label{eq:ac_e_6}  \\
& -\bm {\theta^\Delta}_{ij} \leq \angle (V_i V^*_j) \leq \bm {\theta^\Delta}_{ij} \;\; \forall (i,j) \in E  \label{eq:ac_e_7} 
\end{align}
\end{subequations}
\end{model}

\paragraph{The Extended AC Power Flow Feasibility Problem}

Combining all of these constraints yields the Extended AC Power Flow Feasibility presented in Model \ref{model:ac_pf_e} \refacepf. 
The operational constraints bus voltage and generator output are captured by the variable bounds.
Constraint \eqref{eq:ac_e_0} sets the reference angle, to eliminate numerical symmetries.  
Constraint \eqref{eq:ac_e_1} capture the bus voltage limits.  
Constraints \eqref{eq:ac_e_3} capture KCL and constraints \eqref{eq:ac_e_4}--\eqref{eq:ac_e_5} capture Ohm's Law.  
Finally constraints \eqref{eq:ac_e_6} and \eqref{eq:ac_e_7} enforce the line flow and phase angle difference limits respectively. 
Notice that this is a non-convex nonlinear satisfaction problem due to the product of voltage variables, $V_i V^*_j$ and is NP-Hard \cite{verma2009power,ACSTAR2015}.

Noting that bounds on the voltage variables and edge power flow variables (i.e. $V_i \in ( \bm {V^l}_i, \bm {V^u}_i ), S_{ij} \in (\bm {S^{l}}_{ij},\bm {S^{u}}_{ij})$) are not always specified,
we observe that constraints \eqref{eq:ac_e_1}, \eqref{eq:ac_e_6} respectively imply reasonable bounds on each of these variables.
\begin{lemma}
\label{lemma:flow_lim_const_equiv}
$\bm {S^{u}}_{ij} = \bm {s^u}_{ij} + \bm i \bm {s^u}_{ij}, \bm {S^{l}}_{ij} = -(\bm {s^u}_{ij} + \bm i \bm {s^u}_{ij})$ are valid bounds in \refacepf.
\end{lemma}
\begin{proof}
First, observe that \eqref{eq:ac_e_6} is equivalent to,
\begin{subequations}
\begin{align}
& \Re \left( S_{ij} \right)^2 + \Im \left( S_{ij} \right)^2 \leq (\bm {s^u}_{ij})^2
\end{align}
\end{subequations}
solving for $\Re \left( S_{ij} \right)$ results in,
\begin{subequations}
\begin{align}
\Re \left( S_{ij} \right)^2  &\leq (\bm {s^u}_{ij})^2 - \Im \left( S_{ij} \right)^2 \\
-\sqrt{(\bm {s^u}_{ij})^2 - \Im \left( S_{ij} \right)^2} \leq \Re \left( S_{ij} \right) &\leq \sqrt{(\bm {s^u}_{ij})^2 - \Im \left( S_{ij} \right)^2} 
\end{align}
\end{subequations}
Noticing that $(\bm {s^u}_{ij})^2, \Im(\cdot)^2$ are both positive, we can deduce the following bounds,
\begin{subequations}
\begin{align}
& -\bm {s^u}_{ij} \leq \Re \left( S_{ij} \right) \leq \bm {s^u}_{ij}
\end{align}
\end{subequations}
A similar argument holds for the bounds of $\Im(S_{ij})$, demonstrating the result.
\end{proof}

\begin{lemma}
\label{lemma:flow_lim_const_equiv}
$\bm {V^{u}}_i = \bm {v^u}_{i} + \bm i \bm {v^u}_{i}, \bm {V^{l}}_{ij} = -(\bm {v^u}_{i} + \bm i \bm {v^u}_{i})$ are valid bounds in \refacepf.
\end{lemma}
\begin{proof}
\begin{subequations}
First, observe that the upper bound constraint in \eqref{eq:ac_e_1} is equivalent to,
\begin{align}
& \Re \left( V_{i} \right)^2 + \Im \left( V_{i} \right)^2 \leq (\bm {v^u}_{i})^2
\end{align}
\end{subequations}
The result follows similarly to the previous one.
\end{proof}
\noindent
Also observe that the lower bound constraint in \eqref{eq:ac_e_1} is non-convex and cannot, in general, be incorporated into the bounds on $V$.

\section{Three Relaxations of the Extended AC Power Flow}
\label{sec:relax_e}

This section develops three successive relaxations of \refacepf.  
It begins with the established Second-Order Cone (SOC) relaxation \cite{Jabr06}. 
The SOC relaxation is then further relaxed in to a {\em Network Flow} model (NF), similar to those traditionally studied in operations research.
Lastly, the NF model is then further relaxed into a {\em Copper Plate} model (CP), which ignores all of the network aspects and simply states that the supply must be at least as large as the demand. 

\subsection{The Second-Order Cone Relaxation (SOC)}

The SOC relaxation was first proposed in \cite{Jabr06} and utilizes two key insights.  First, by lifting the product of voltage variables $V_i V_j^*$ into a higher dimensional space (i.e. the $W$-space),
\begin{subequations}
\begin{align}
& W_{i} = |V_{i}|^2 \;\; i \in N \label{eq:w_link_1} \\
& W_{ij} = V_i V_j^* \;\; \forall(i,j) \in E \label{eq:w_link_2} 
\end{align}
\end{subequations}
a linear relaxation of \refacepf~is obtained.  Second, a relaxation of the absolute square of the voltage products is developed to strengthen this $W$-space relaxation, as follows,
\begin{subequations}
\begin{align}
\left( \frac{V_i}{\bm T_{ij}^*} V_j^* \right) \left( \frac{V_i}{\bm T_{ij}^*} V_j^* \right)^* &= \left( \frac{V_i}{\bm T_{ij}^*} V_j^* \right) \left( \frac{V_i^*}{\bm T_{ij}} V_j \right) \\ 
\frac{|V_i  V_j^*|^2}{|\bm T_{ij}^*|^2} &= \frac{|V_i|^2}{|\bm T_{ij}^*|^2} |V_j|^2 \\
|V_iV_j^*|^2 &= |V_i|^2 |V_j|^2 \\
|W_{ij}|^2 &= W_i W_j \;\; \forall(i,j) \in E \\
|W_{ij}|^2 &\leq W_i W_j \;\; \forall(i,j) \in E \label{eq:asvp_w}
\end{align}
\end{subequations}
Notice that constraint \eqref{eq:asvp_w} is a convex second-order cone constraint, which is widely supported by industrial strength convex optimization tools (e.g., Gurobi \cite{gurobi}, CPlex \cite{cplex}, Mosek \cite{mosek}).

The complete SOC relaxation of \refacepf~is presented in Model \ref{model:soc_pf_e} \refsocepf.
Constraints \eqref{eq:soc_e_3} capture KCL and constraints \eqref{eq:soc_e_4}--\eqref{eq:soc_e_5} capture line power flow in the $W$-space.  
Constraints \eqref{eq:soc_e_6} strengthen the relaxation with voltage product second-order cone constraint, and 
constraints \eqref{eq:soc_e_8}--\eqref{eq:soc_e_2} capture the line power flow and phase angle difference operational constraints.

\begin{model}[t]
\caption{The SOC Relaxation of Extended AC Power Flow (SOC-E-PF)}
\label{model:soc_pf_e}
\begin{subequations}
\begin{align}
\mbox{\bf variables:} \nonumber \\
& S^g_i \in ( \bm {S^{gl}}_i, \bm {S^{gu}}_i) \;\; \forall i\in N \nonumber \\
& W_i \in ( (\bm {v^l}_i)^2, (\bm {v^u}_i)^2 ) \;\; \forall i\in N \nonumber \\
& W_{ij} \;\; \forall (i,j)\in E \nonumber \\
& S_{ij} \in (\bm {S^{l}}_{ij},\bm {S^{u}}_{ij})\;\; \forall (i,j),(j,i) \in E \nonumber \\
%
%
\mbox{\bf subject to:} \nonumber \\
%
& S^g_i - {\bm S^d_i} - (\bm Y^s_i)^* W_i = \sum_{\substack{(i,j)\in E}} S_{ij} + \sum_{\substack{(j,i)\in E}} S_{ji} \;\; \forall i\in N \label{eq:soc_e_3} \\ 
& S_{ij} = \left( \bm Y^*_{ij} - \bm i\frac{\bm {b^c}_{ij}}{2} \right)\frac{W_i}{|\bm{T}_{ij}|^2}  - \bm Y^*_{ij}\frac{W_{ij}}{\bm{T}^*_{ij}}  \;\; (i,j)\in E \label{eq:soc_e_4}  \\
& S_{ji} =  \left( \bm Y^*_{ij} - \bm i\frac{\bm {b^c}_{ij}}{2} \right) W_j - \bm Y^*_{ij}\frac{W_{ij}^*}{\bm{T}_{ij}} \;\; (i,j)\in E \label{eq:soc_e_5} \\
& |W_{ij}|^2 \leq W_i W_j  \;\; \forall (i,j) \in E \label{eq:soc_e_6} \\
& |S_{ij}| \leq \bm {s^u}_{ij} \;\; \forall (i,j),(j,i) \in E \label{eq:soc_e_8} \\
& \tan(-\bm {\theta^\Delta}_{ij}) \Re\left(W_{ij}\right) \leq \Im\left(W_{ij}\right) \leq \tan(\bm {\theta^\Delta}_{ij}) \Re\left(W_{ij}\right) \;\; \forall (i,j) \in E \label{eq:soc_e_2} 
\end{align}
\end{subequations}
\end{model}

\begin{lemma}
\label{lemma:soc_relax}
\refsocepf~is a relaxation of \refacepf.
\end{lemma}
\begin{proof}
A reduction from the SDP relaxation was first observed in \cite{6345272}.
\end{proof}

\subsection{The Network Flow Relaxation (NF)}

The inspiration for the network flow relaxation is to produce a
relaxation that is similar to the linear flow models that are widely
studied in operations research and computer science
\cite{ahuja1993network}.  This section proposes Model
\ref{model:nf_pf_e} \refnfepf~as such an analogue.  The bulk of
  the operational constraints in this model are identical to
  \refsocepf, however the convex non-linear thermal limit constraint
  \eqref{eq:soc_e_8} is omitted in the interest of linearity.  The key
  deference in this model is that the line flow constraints are
  simplified in to the linear expression \eqref{eq:nf_e_3}, and no
  longer require the $W_{ij}$ variables.
  The resulting model is essentially a traditional network flow \cite{ahuja1993network} with two additional constraints to capture the line losses \eqref{eq:nf_e_3} and the phase angle differences \eqref{eq:nf_e_1}--\eqref{eq:nf_e_2}. 
  Observe that, when there is no line charging, constraints \eqref{eq:nf_e_3} simply express that the line losses must be nonnegative, i.e. power cannot be created on a line.

\begin{model}[t]
\caption{The Network Flow Relaxation of the Extended AC Power Flow (NF-E-PF)}
\label{model:nf_pf_e}
\begin{subequations}
\begin{align}
\mbox{\bf variables:} \nonumber \\
& S^g_i \in ( \bm {S^{gl}}_i, \bm {S^{gu}}_i) \;\; \forall i\in N \nonumber \\
& W_i \in ( (\bm {v^l}_i)^2, (\bm {v^u}_i)^2 ) \;\; \forall i\in N \nonumber \\
& S_{ij} \in (\bm {S^{l}}_{ij},\bm {S^{u}}_{ij})\;\; \forall (i,j),(j,i) \in E \nonumber \\
%
%
\mbox{\bf subject to:} \nonumber \\
%
%
& S^g_i - {\bm S^d_i} - (\bm Y^s_i)^* W_i = \sum_{\substack{(i,j)\in E}} S_{ij} + \sum_{\substack{(j,i)\in E}} S_{ji} \;\; \forall i\in N \label{eq:nf_e_7} \\
& S_{ij} +  S_{ji} \geq -\bm i \frac{\bm {b^c}_{ij}}{2} \left( \frac{W_i}{|\bm{T}_{ij}|^2} + W_j \right) \;\; (i,j)\in E \label{eq:nf_e_3} \\
%
& \tan(-\bm {\theta^\Delta}_{ij}) \Re\left( \bm Z^*_{ij}  \bm{T}^*_{ij} \left( \left( \bm Y^*_{ij} - \bm i\frac{\bm {b^c}_{ij}}{2} \right) \frac{W_i}{|\bm T_{ij}|^2} - S_{ij} \right) \right) \leq \nonumber \\ 
& \hspace{4.5cm} \Im\left( \bm Z^*_{ij}  \bm{T}^*_{ij} \left( \left( \bm Y^*_{ij} - \bm i\frac{\bm {b^c}_{ij}}{2} \right) \frac{W_i}{|\bm T_{ij}|^2} - S_{ij} \right)  \right) \;\; \forall (i,j) \in E \label{eq:nf_e_1}  \\
& \tan(\bm {\theta^\Delta}_{ij}) \Re\left( \bm Z^*_{ij}  \bm{T}^*_{ij} \left( \left( \bm Y^*_{ij} - \bm i\frac{\bm {b^c}_{ij}}{2} \right) \frac{W_i}{|\bm T_{ij}|^2} - S_{ij} \right) \right) \geq \nonumber \\ 
& \hspace{4.5cm} \Im\left( \bm Z^*_{ij}  \bm{T}^*_{ij} \left( \left( \bm Y^*_{ij} - \bm i\frac{\bm {b^c}_{ij}}{2} \right) \frac{W_i}{|\bm T_{ij}|^2} - S_{ij} \right)  \right) \;\; \forall (i,j) \in E \label{eq:nf_e_2}  
\end{align}
\end{subequations}
\end{model}

In the rest of this section, we demonstrate that \refnfepf~is a linear relaxation of \refsocepf.
Through a series of deductions we will show that \eqref{eq:nf_e_3} is simply a weaker version of \eqref{eq:soc_e_4}--\eqref{eq:soc_e_6}, demonstrating the relaxation property between these two models.
We begin by observing that the line losses in \refsocepf~are given by,
\begin{subequations}
\begin{align}
&  S_{ij} +  S_{ji} = 
\bm Y^*_{ij} \left( \frac{W_i}{ |\bm{T}_{ij}|^2} - \frac{W_{ij}}{\bm{T}^*_{ij}} - \frac{W^*_{ij}}{\bm{T}_{ij}} + W_j \right)
-  \bm i \frac{\bm {b^c}_{ij}}{2} \left(   \frac{W_i}{ |\bm{T}_{ij}|^2} + W_j \right) \label{eq:w_loss_e}
\end{align}
\end{subequations}

\begin{lemma}
\label{lemma:cm_pos}
\eqref{eq:asvp_w} ensures $ W_i / |\bm{T}_{ij}|^2 - W_{ij}/\bm{T}^*_{ij} - W^*_{ij}/\bm{T}_{ij} + W_j \geq 0$.
\end{lemma}
\begin{proof}
First, observe that \eqref{eq:asvp_w} is equivalent to,
\begin{subequations}
\begin{align}
& \left| \frac{W_{ij}}{\bm{T}^*_{ij}} \right|^2 \leq \frac{W_i}{|\bm T_{ij}|^2} W_j \label{eq:asvp_w_t}
\end{align}
\end{subequations}
which is equivalent to the following,
\begin{subequations}
\begin{align}
 \Re \left( \frac{W_{ij}}{\bm{T}^*_{ij}} \right)^2 + \Im \left( \frac{W_{ij}}{\bm{T}^*_{ij}} \right)^2 &\leq \frac{W_i}{|\bm T_{ij}|^2} W_j \\
 \Re \left( \frac{W_{ij}}{\bm{T}^*_{ij}} \right)^2 &\leq \frac{W_i}{|\bm T_{ij}|^2} W_j - \Im \left( \frac{W_{ij}}{\bm{T}^*_{ij}} \right)^2 \\
 -\sqrt{\frac{W_i}{|\bm T_{ij}|^2} W_j - \Im \left( \frac{W_{ij}}{\bm{T}^*_{ij}} \right)^2} \leq \Re \left( \frac{W_{ij}}{\bm{T}^*_{ij}} \right) &\leq \sqrt{\frac{W_i}{|\bm T_{ij}|^2} W_j - \Im \left( \frac{W_{ij}}{\bm{T}^*_{ij}} \right)^2}
\end{align}
\end{subequations}
Noticing that $W_i, W_j, \Im(\cdot)^2, |\bm T_{ij}|^2$ are all positive, we can deduce the following bounds,
\begin{subequations}
\begin{align}
& -\sqrt{\frac{W_i}{|\bm T_{ij}|^2} W_j}  \leq  \Re \left( \frac{W_{ij}}{\bm{T}^*_{ij}} \right) \leq \sqrt{\frac{W_i}{|\bm T_{ij}|^2} W_j} \label{eq:w_r_bounds}
\end{align}
\end{subequations}
Now observe the equivalence,
\begin{subequations}
\begin{align}
& \frac{W_i}{ |\bm{T}_{ij}|^2} - \frac{W_{ij}}{\bm{T}^*_{ij}} - \frac{W^*_{ij}}{\bm{T}_{ij}} + W_j \Leftrightarrow  \frac{W_i}{ |\bm{T}_{ij}|^2} + W_j - 2 \Re \left( \frac{W_{ij}}{\bm{T}^*_{ij}} \right) \label{eq:w_curr}  
\end{align}
\end{subequations}
We want to determine the smallest possible value of \eqref{eq:w_curr}. Given that $W_i, W_j$ are strictly positive, the largest possible value of $\Re(W_{ij} / \bm{T}^*_{ij})$ will minimize \eqref{eq:w_curr}.  Utilizing \eqref{eq:w_r_bounds} as a bound for this value, we have that smallest possible value of \eqref{eq:w_curr} no smaller than,
\begin{subequations}
\begin{align}
& \frac{W_i}{ |\bm{T}_{ij}|^2} + W_j - 2 \sqrt{\frac{W_i}{|\bm T_{ij}|^2} W_j} 
\end{align}
\end{subequations}
which factors to,
\begin{subequations}
\begin{align}
& \left( \frac{\sqrt{W_i}}{|\bm{T}_{ij}|} - \sqrt{W_j} \right)^2
\end{align}
\end{subequations}
Given that the square of any expression is a positive number, the result follows.
\end{proof}

\begin{theorem}
\label{theorem:main_nf}
\eqref{eq:nf_e_3} is a valid redundant constraint in \refsocepf~when $\bm Z_{ij} \geq 0$.
\end{theorem}
\begin{proof}
$\bm Z_{ij} \geq 0 \Leftrightarrow \bm Y^*_{ij} \geq 0$, combining this fact with Lemma \ref{lemma:cm_pos} yields,
\begin{subequations}
\begin{align}
&  \bm Y^*_{ij} \left( \frac{W_i}{ |\bm{T}_{ij}|^2} - \frac{W_{ij}}{\bm{T}^*_{ij}} - \frac{W^*_{ij}}{\bm{T}_{ij}} + W_j \right) \geq 0 \label{eq:curr_mag_pos}
\end{align}
\end{subequations}
The result follows from combining \eqref{eq:curr_mag_pos} with the definition of line loss \eqref{eq:w_loss_e}.
\end{proof}

\begin{lemma}
\label{lemma:pad_const_equiv}
\eqref{eq:nf_e_1}--\eqref{eq:nf_e_2} are valid redundant constraints in \refsocepf.
\end{lemma}
\begin{proof}
Begin by factoring \eqref{eq:soc_e_4} similarly to \eqref{eq:vv_factor_e}.
Then it is easy to observe that \eqref{eq:nf_e_1}--\eqref{eq:nf_e_2} can be derived by combining \eqref{eq:soc_e_2}--\eqref{eq:soc_e_4} in \refsocepf.
\end{proof}

\begin{corollary}
\label{corollary:model}
\refnfepf~is a relaxation of \refsocepf.
\end{corollary}

\subsection{The Copper Plate Relaxation (CP)}

The inspiration for this relaxation is to develop a version of the
classic copper plate approximations used in power system analysis,
which simply state that supply and demand should be balanced
throughout the network.  This section proposes Model
\ref{model:cp_pf_e} \refcpepf~as {\em the relaxation version of a
  copper plate approximation}.  All of the operational constraints in
this model are identical to \refnfepf, however the constraints and
variables relating line flows ($S_{ij}$) have been eliminated.
Additionally, all of the KCL constraints have been combined into one
simple linear expression \eqref{eq:cp_e_1}.  A key benefit of this
relaxation is that it is incredibly simple and scalable, while
  still capturing some properties of the lines.
In the rest of this section we demonstrate that \refcpepf~is a relaxation of \refnfepf.

\begin{model}[t]
\caption{The Copper Plate Relaxation of the Extended AC Power Flow (CP-E-PF)}
\label{model:cp_pf_e}
\begin{subequations}
\begin{align}
\mbox{\bf variables:} \nonumber \\
& S^g_i \in ( \bm {S^{gl}}_i, \bm {S^{gu}}_i) \;\; \forall i\in N \nonumber \\
& W_i \in ( (\bm {v^l}_i)^2, (\bm {v^u}_i)^2 ) \;\; \forall i\in N \nonumber \\
%
%
\mbox{\bf subject to:} \nonumber \\
%
& \sum_{i \in N} S^g_i - \sum_{i \in N} {\bm S^d_i} - \sum_{i \in N} (\bm Y^s_i)^* W_i \geq \sum_{\substack{(i,j)\in E }}  -\bm i \frac{\bm {b^c}_{ij}}{2} \left( \frac{W_i}{|\bm{T}_{ij}|^2} + W_j \right) \label{eq:cp_e_1}
%
\end{align}
\end{subequations}
\end{model}


\begin{theorem}
\label{theorem:main_nf}
\eqref{eq:cp_e_1} is a valid redundant constraint in \refnfepf.
\end{theorem}
\begin{proof}
Due to the conjunctive nature of mathematical programs, combining constraints yields redundant constraints.  
As a first step, all of the KCL constraints \eqref{eq:soc_e_3} are combined together.
\begin{subequations}
\begin{align}
& \sum_{i \in N} \left( S^g_i - {\bm S^d_i} - (\bm Y^s_i)^* W_i = \sum_{\substack{(i,j)\in E}} S_{ij} + \sum_{\substack{(j,i)\in E}} S_{ji} \right) \\
& \sum_{i \in N} S^g_i -  \sum_{i \in N}  {\bm S^d_i} -  \sum_{i \in N}  (\bm Y^s_i)^* W_i = \sum_{i \in N} \left(  \sum_{\substack{(i,j)\in E}} S_{ij} + \sum_{\substack{(j,i)\in E}} S_{ji} \right) \\
& \sum_{i \in N} S^g_i -  \sum_{i \in N}  {\bm S^d_i} -  \sum_{i \in N}  (\bm Y^s_i)^* W_i = \sum_{\substack{(i,j)\in E}} S_{ij} + S_{ji}
\end{align}
\end{subequations}
Observe that line loss constraints \eqref{eq:nf_e_3} can be used to eliminate the $S$ variables from this expression entirely, yielding,
\begin{subequations}
\begin{align}
& \sum_{i \in N} S^g_i -  \sum_{i \in N}  {\bm S^d_i} -  \sum_{i \in N}  (\bm Y^s_i)^* W_i \geq \sum_{\substack{(i,j)\in E}}  -\bm i \frac{\bm {b^c}_{ij}}{2} \left( \frac{W_i}{|\bm{T}_{ij}|^2} + W_j \right)
\end{align}
\end{subequations}
which demonstrates the result.
\end{proof}

\begin{corollary}
\label{corollary:model}
\refcpepf~is a relaxation of \refnfepf.
\end{corollary}

\noindent
Note \refcpepf~includes all of the assumptions of \refnfepf~and hence can only be applied when $\bm Z_{ij} \geq 0 \;\forall (i,j) \in E$.

\section{Comparison of Linear Relaxations}
\label{sec:lin_relax}

One of the key benefits of both \refnfepf~and \refcpepf~is that they are linear relaxations of \refacepf.
Although linear approximations of \refacepf~are quite common \cite{Stott_2009bb,LPAC_ijoc,IVModel}, linear {\em relaxations} of \refacepf~are, to the best of our knowledge, limited to two previous works \cite{5772044,linear_opf}, which we review in detail.  

There are two key differences of this work and \cite{linear_opf}.  First, \cite{linear_opf} focuses on only the rectangular real number formulation of \refacepf~and all of the proofs follow from that formulation.  This work uses complex numbers directly for the proofs, hence the results apply to any real number realization of these models.  Second, the primary goal of \cite{linear_opf} is to propose valid linear inequalities of non-convex constraints for iteratively strengthening relaxations of the \refacepf, while this work focuses on building a intuitive static model.  Indeed, the results of \cite{linear_opf} may be used in consort with this work, to iteratively strengthen the linear relaxations proposed herein.

The most closely related work to this is \cite{5772044}, which proposes a relaxation along the lines of \refnfepf.  The key difference is that instead of constraints \eqref{eq:nf_e_3}, \cite{5772044} uses the following valid equalities,
\begin{subequations}
\begin{align}
\Re\left( (\bm b_{ij} + \bm i \bm g_{ij}) \left(S_{ij} + S_{ji} \right) \right) &= -\bm g_{ij}\frac{\bm {b^c}_{ij}}{2} \left( \frac{W_i}{|\bm T_{ij}|^2}+W_j \right) \;\; \forall (i,j) \in E \\
\Re\left( \bm Y_{ij} \left(S_{ij} - S_{ji} \right) \right) &= \left(|\bm Y_{ij}|^2 + \bm b_{ij}\frac{\bm {b^c}_{ij}}{2} \right) \left( \frac{W_i}{|\bm T_{ij}|^2}-W_j \right)  \;\; \forall (i,j) \in E
\end{align}
\end{subequations}
Because these constraints are equalities (not inequalities), it is difficult to develop a strong theoretical connection to the models proposed herein.  However, as will be demonstrated in an example, the linear relaxation proposed in \cite{5772044} can generate power on the lines, which is a significant disadvantage over the models proposed here.
In the rest of of the paper TH will be used to denote the linear relaxation proposed in \cite{5772044}.

\subsection{An Illustrative Example}
\label{sec:example}

\begin{figure}[t]
\center
    \includegraphics[width=4.0cm]{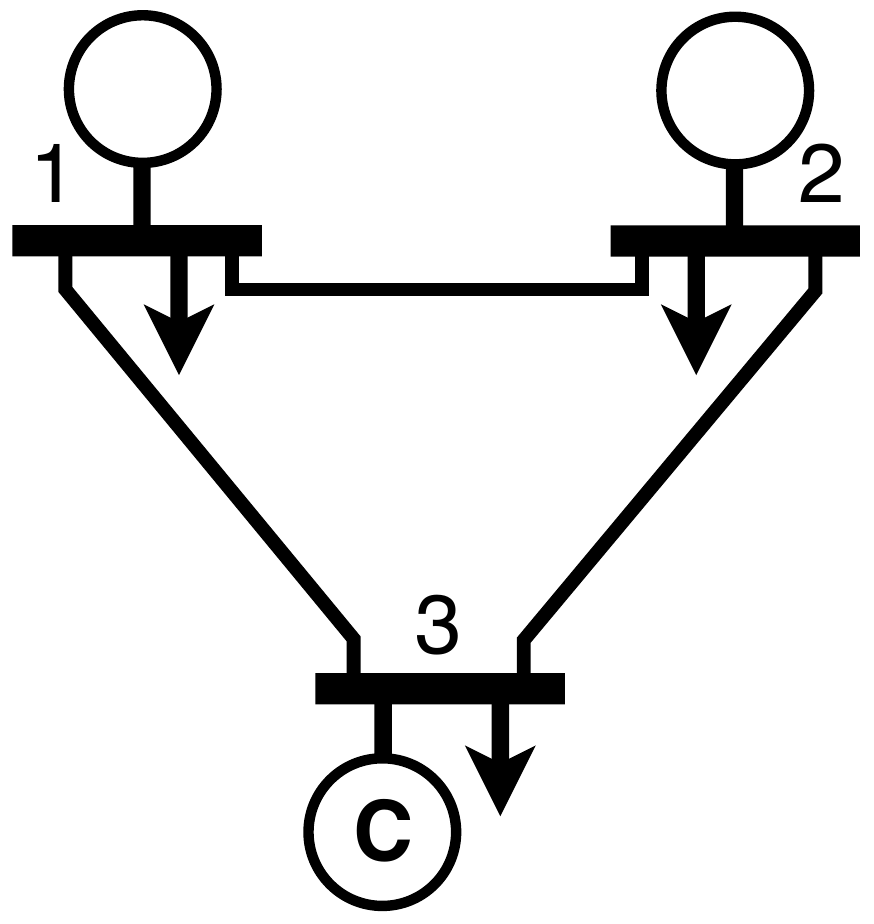} 
\caption{3-Bus Example Network Diagram.}
\label{fig:3_bus_example}
\end{figure}

\begin{table}[t]
\centering
\begin{tabular}{|c||c|c|c|c|}
\hline
\multicolumn{5}{|c|}{Bus Parameters} \\
\hline
Bus & $\boldsymbol {p^d}$ & $\boldsymbol {q^d}$ & $\boldsymbol {v^l}$ & $\boldsymbol {v^u}$ \\
\hline
1 & 110 & 40 & 0.9 & 1.1 \\ 
\hline 
2 & 110 & 40 & 0.9 & 1.1 \\ 
\hline 
3 & 95 & 50 & 0.9 & 1.1 \\ 
\hline 
\end{tabular}
\vspace{0.5cm}
\begin{tabular}{|c||c|c|c|c|c|}
\hline
\multicolumn{6}{|c|}{Line Parameters} \\
\hline
From--To Bus & $\boldsymbol {r}$ & $\boldsymbol {x}$ & $\boldsymbol {b^c}$ & $\boldsymbol {s^u}$ & $\bm {\theta^\Delta}$ \\
\hline
1--2 & 0.042 & 0.90 & 0.30 & $\infty$ & $30^\circ$ \\ 
\hline 
2--3 & 0.025 & 0.75 & 0.70 & 50 & $30^\circ$ \\ 
\hline 
1--3 & 0.065 & 0.62 & 0.45 & $\infty$ & $30^\circ$ \\ 
\hline 
\end{tabular}
\begin{tabular}{|c||c|c||c|c|c|}
\hline
\multicolumn{6}{|c|}{Generator Parameters} \\
\hline
Generator & $\boldsymbol {p^{gl}}, \boldsymbol {p^{gu}}$ & $\boldsymbol {q^{gl}}, \boldsymbol {q^{gu}}$ & $\boldsymbol {c_2}$ & $\boldsymbol {c_1}$ & $\boldsymbol {c_0}$ \\
\hline
1 & $0,\infty$ & $-\infty,\infty$ & 0.110 & 5.0 & 0 \\ 
\hline 
2 & $0,\infty$ & $-\infty,\infty$ & 0.085 & 1.2 & 0 \\ 
\hline 
3 & $0,0$ & $-\infty,\infty$ & 0 & 0 & 0 \\ 
\hline 
\end{tabular}
\caption{Three-Bus System Network Data (100 MVA Base).}
\label{tbl:3_bus_network_data}
\end{table}

This section illustrates all of the power flow relaxations considered
here on the 3-bus network from \cite{6120344}, which has proven to be
an excellent test case for power flow relaxations. This system is
depicted in Figure \ref{fig:3_bus_example} and the associated network
parameters are given in Table \ref{tbl:3_bus_network_data}.  This
network is designed to have very few binding constraints.  Hence, the
generator and line limits are set to large non-binding values, except
for the thermal limit constraint on the line between buses 2 and 3,
which is set to 50 MVA. In addition to its base configuration, we also
consider this network with reduced phase angle difference bounds of
$18$ degrees.  The nonlinear global optimization solver {\sc Couenne}
\cite{belotti2009couenne} is used to find the globally optimal
solution to \refacepf~and the {\em optimally gap} between this
solution and a relaxation is computed using the formula,
\begin{align}
\frac{\text{Heuristic} - \text{Relaxation}}{\text{Heuristic}}. \nonumber
\end{align}
Hence, the smaller the optimality gap the better the relaxation.

\begin{table}[t]
\centering
\begin{tabular}{|r||r||r|r|r|r|r|r|r|r|r||r|r|r|r|c|c|}
\hline
                 & \$/h & \multicolumn{4}{c|}{Optimality Gap (\%)}  \\
                 &       & \cite{Jabr06} \hspace{0.05cm} &       &       &\cite{5772044} \\
Test Case & AC & SOC & NF & CP & TH \\
\hline  
\hline  
Base & {\bf 5812} & 1.32 & 2.99 & 2.99 & 87.2 \\
\hline 
 $ {\theta^\Delta} \! = \!18^\circ$ & {\bf 5992} & 4.28 & 5.90 & 5.90 & 87.5 \\
\hline 
\end{tabular}
\caption{AC-OPF Bounds using Relaxations on the 3-Bus Case.}
\label{tbl:3_bus_results}
\end{table}

Table \ref{tbl:3_bus_results} summarizes the results. As expected,
in both cases the SOC relaxation has the smallest optimality gap. 
Further relaxing the model to NF or CP, increases the gap by about $1.5\%$.
In this example the NF and CP produce the same optimally gap, however that is not the case in general.
The TH relaxation has an optimality gap of over 80\% indicating that it is significantly weaker than the 
other linear relaxations considered here.

\section{Conclusion}
\label{sec:conc}

This paper has developed two intuitive linear relaxations of AC power flow for transmission systems, one based on network flows (NF) and another inspired by copper plate approximations (CP).
Although it was shown that both models are weaker than the established nonlinear SOC relaxation, their linearity brings significant performance and scalability benefits.
Combining these linear relaxations with the established nonlinear relaxations (SDP \cite{Bai2008383}, QC \cite{QCarchive}, SOC \cite{Jabr06}) provides a rich variety of tradeoffs between the relaxation accuracy and scalability.
Given that all of these power flow relaxations span a wide range of mathematical programs and associated solution methods, the natural frontier for future work is to conduct a detailed numerical study to better understand the time and quality tradeoff of each relaxation.

\section{Acknowledgements}

NICTA is funded by the Australian Government through the Department of
Communications and the Australian Research Council through the ICT
Centre of Excellence Program.

\bibliographystyle{plain}
\bibliography{../power_models}

\begin{thebibliography}{10}

\bibitem{ahuja1993network}
R.K. Ahuja, T.L. Magnanti, and J.B. Orlin.
\newblock {\em Network flows: theory, algorithms, and applications}.
\newblock Prentice Hall, 1993.

\bibitem{Bai2008383}
Xiaoqing Bai, Hua Wei, Katsuki Fujisawa, and Yong Wang.
\newblock Semidefinite programming for optimal power flow problems.
\newblock {\em International Journal of Electrical Power \& Energy Systems},
  30(6–7):383 -- 392, 2008.

\bibitem{belotti2009couenne}
Pietro Belotti.
\newblock Couenne: User manual.
\newblock Published online at \url{https://projects.coin-or.org/Couenne/},
  2009.
\newblock Accessed: 10/04/2015.

\bibitem{linear_opf}
Daniel Bienstock and Gonzalo Munoz.
\newblock {On linear relaxations of OPF problems}.
\newblock {\em CoRR}, abs/1411.1120, 2014.

\bibitem{LPAC_ijoc}
Coffrin Carleton and Pascal Van~Hentenryck.
\newblock A linear-programming approximation of ac power flows.
\newblock {\em Forthcoming in INFORMS Journal on Computing}, 2014.

\bibitem{6102366}
M.~Farivar, C.R. Clarke, S.H. Low, and K.M. Chandy.
\newblock Inverter var control for distribution systems with renewables.
\newblock In {\em 2011 IEEE International Conference on Smart Grid
  Communications (SmartGridComm)}, pages 457--462, Oct 2011.

\bibitem{gurobi}
{Gurobi Optimization, Inc.}
\newblock Gurobi optimizer reference manual.
\newblock Published online at \url{http://www.gurobi.com}, 2014.

\bibitem{QCarchive}
H.~Hijazi, C.~Coffrin, and P.~Van~Hentenryck.
\newblock Convex quadratic relaxations of mixed-integer nonlinear programs in
  power systems.
\newblock Published online at
  \url{http://www.optimization-online.org/DB_HTML/2013/09/4057.html}, 2013.

\bibitem{cplex}
Inc. IBM.
\newblock {IBM ILOG CPLEX Optimization Studio}.
\newblock
  \url{http://www-01.ibm.com/software/commerce/optimization/cplex-optimizer/},
  2014.

\bibitem{Jabr06}
R.A. Jabr.
\newblock Radial distribution load flow using conic programming.
\newblock {\em IEEE Transactions on Power Systems}, 21(3):1458--1459, Aug 2006.

\bibitem{9780070359581}
Prabha Kundur.
\newblock {\em Power System Stability and Control}.
\newblock McGraw-Hill Professional, 1994.

\bibitem{5971792}
J.~Lavaei and S.H. Low.
\newblock Zero duality gap in optimal power flow problem.
\newblock {\em IEEE Transactions on Power Systems}, 27(1):92 --107, feb. 2012.

\bibitem{ACSTAR2015}
Karsten Lehmann, Alban Grastien, and Pascal Van~Hentenryck.
\newblock {AC-Feasibility on Tree Networks is NP-Hard}.
\newblock {\em IEEE Transactions on Power Systems}, 2015 (to appear).

\bibitem{6120344}
B.C. Lesieutre, D.K. Molzahn, A.R. Borden, and C.L. DeMarco.
\newblock Examining the limits of the application of semidefinite programming
  to power flow problems.
\newblock In {\em 49th Annual Allerton Conference on Communication, Control,
  and Computing (Allerton), 2011}, pages 1492 --1499, sept. 2011.

\bibitem{6822653}
R.~Madani, S.~Sojoudi, and J.~Lavaei.
\newblock Convex relaxation for optimal power flow problem: Mesh networks.
\newblock {\em IEEE Transactions on Power Systems}, 30(1):199--211, Jan 2015.

\bibitem{7038397}
D.K. Molzahn and I.A. Hiskens.
\newblock Moment-based relaxation of the optimal power flow problem.
\newblock In {\em Power Systems Computation Conference (PSCC), 2014}, pages
  1--7, Aug 2014.

\bibitem{6980142}
D.K. Molzahn and I.A. Hiskens.
\newblock Sparsity-exploiting moment-based relaxations of the optimal power
  flow problem.
\newblock {\em Power Systems, IEEE Transactions on}, PP(99):1--13, 2014.

\bibitem{IVModel}
Richard~P. O'Neill, Anya Castillo, and Mary~B. Cain.
\newblock The iv formulation and linear approximations of the ac optimal power
  flow problem.
\newblock Published online at
  \url{http://www.ferc.gov/industries/electric/indus-act/market-planning/opf-papers/acopf-2-iv-linearization.pdf},
  December 2012.
\newblock Accessed: 18/11/2013.

\bibitem{Purchala:2005gt}
K~Purchala, L~Meeus, D~Van~Dommelen, and R~Belmans.
\newblock {Usefulness of {DC} power flow for active power flow analysis}.
\newblock {\em Power Engineering Society General Meeting}, pages 454--459,
  2005.

\bibitem{6345272}
S.~Sojoudi and J.~Lavaei.
\newblock Physics of power networks makes hard optimization problems easy to
  solve.
\newblock In {\em Power and Energy Society General Meeting, 2012 IEEE}, pages
  1--8, July 2012.

\bibitem{Stott_2009bb}
B~Stott, J~Jardim, and O~Alsac.
\newblock Dc power flow revisited.
\newblock {\em IEEE Transactions on Power Systems}, 24(3):1290--1300, 2009.

\bibitem{5772044}
J.A. Taylor and F.S. Hover.
\newblock Linear relaxations for transmission system planning.
\newblock {\em IEEE Transactions on Power Systems}, 26(4):2533 --2538, nov.
  2011.

\bibitem{mosek}
K.~C. Toh, R.~H. TŸtŸncŸ, and M.~J. Todd.
\newblock {SDPT3 - a MATLAB software package for semidefinite-quadratic-linear
  programming}.
\newblock \url{https://mosek.com/}, 2014.

\bibitem{verma2009power}
Abhinav Verma.
\newblock {\em Power grid security analysis: An optimization approach}.
\newblock PhD thesis, Columbia University, 2009.

\bibitem{matpower}
R.D. Zimmerman, C.E. Murillo-S‡~andnchez, and R.J. Thomas.
\newblock Matpower: Steady-state operations, planning, and analysis tools for
  power systems research and education.
\newblock {\em IEEE Transactions on Power Systems}, 26(1):12 --19, feb. 2011.

\end{thebibliography}

\appendix

\section{Real Number Formulations}

In the interest of clean proofs, this document focuses on a complex number representation of the relaxations.
In practice however, the a real number representation is often needed for implementation.
This section presents the real number version of \refnfepf~and \refcpepf~to aid the implementations of these models.

First we develop a real number short-hand for $\bm Z_{ij} \bm{T}_{ij}$ as follows,
\begin{subequations}
\begin{align}
& \Re \left( \bm Z_{ij} \bm{T}_{ij} \right) = \bm {tz^R}_{ij}  = \bm r_{ij} \bm {t^R}_{ij} -\bm x_{ij} \bm {t^I}_{ij} \;\; (i,j) \in E \\
& \Im \left( \bm Z_{ij} \bm{T}_{ij} \right) = \bm {tz^I}_{ij}  = \bm r_{ij} \bm {t^I}_{ij} + \bm x_{ij} \bm {t^R}_{ij} \;\; (i,j) \in E
\end{align}
\end{subequations}
Notice that the superscripts $R,I$ are used for the real and imaginary parts of a complex number respectively.
Expanding \eqref{eq:vv_factor_e} using this short-hand into real and imaginary components $\Re(V_iV^*_j), \Im(V_iV^*_j)$ we have,
\begin{subequations}
\begin{align}
& \Re(V_iV^*_j) = \left( \bm {t^R}_{ij} - \bm {tz^I}_{ij} \left( \frac{ \bm {b^c}_{ij}}{2} \right) \right) \frac{v^2_i}{\bm t^2_{ij}} - \bm {tz^R}_{ij}p_{ij} - \bm {tz^I}_{ij}q_{ij}   \;\; (i,j) \in E \\
& \Im(V_iV^*_j) = \left( -\bm {t^I}_{ij} - \bm {tz^R}_{ij} \left( \frac{ \bm {b^c}_{ij}}{2} \right) \right) \frac{v^2_i}{\bm t^2_{ij}} - \bm {tz^R}_{ij}q_{ij} + \bm {tz^I}_{ij}p_{ij}   \;\; (i,j) \in E
\end{align}
\end{subequations}
both of which are used to implement the phase angle difference constraints in any model with $v,p,q$ variables.  
The complete real number implementation of \refnfepf~and \refcpepf~are presented in Model \ref{model:nf_pf_e_real} and Model \ref{model:cp_pf_e_real} respectively.

\begin{model}[t]
\caption{The Network Flow Relaxation of the Extended AC Power Flow  \refnfepf~in Real Numbers.}
\label{model:nf_pf_e_real}
\begin{subequations}
\begin{align}
\mbox{\bf variables: \hspace{-1.3cm} } \nonumber \\
& p^g_i \in (\bm {p^{gl}}_i, \bm {p^{gu}}_i) \;\; \forall i\in N \nonumber \\
& q^g_i \in (\bm {q^{gl}}_i, \bm {q^{gu}}_i) \;\; \forall i\in N \nonumber \\
& w_i \in ((\bm {v^l}_i)^2, (\bm {v^u}_i)^2) \;\; \forall i\in N \nonumber \\
& p_{ij} \in (-\bm {s^u}_{ij}, \bm {s^u}_{ij}) \;\; \forall (i,j),(j,i) \in E \nonumber \\
& q_{ij} \in (-\bm {s^u}_{ij}, \bm {s^u}_{ij}) \;\; \forall (i,j),(j,i) \in E \nonumber \\
\mbox{\bf subject to: \hspace{-1.3cm} } \nonumber \\
%
%
%
& p^g_i - {\bm p^d_i} - \bm g^s_{i} w_i = \sum_{\substack{(i,j)\in E}} p_{ij} + \sum_{\substack{(j,i)\in E}} p_{ji} \;\; \forall i\in N \\ 
& q^g_i - {\bm q^d_i} + \bm b^s_{i} w_i = \sum_{\substack{(i,j)\in E}} q_{ij} + \sum_{\substack{(j,i)\in E}} q_{ji} \;\; \forall i\in N \\ 
&p_{ij} + p_{ji} \geq 0 \;\; (i,j) \in E \\
&q_{ij} + q_{ji} \geq -\frac{\bm {b^c}_{ij}}{2} \left( \frac{w_i}{\bm {t}^2_{ij}} +  w_j \right) \;\; (i,j) \in E \\
%
%
& \tan(-\bm {\theta^\Delta}) \left( \left( \bm {t^R}_{ij} - \bm {tz^I}_{ij} \left( \frac{ \bm {b^c}_{ij}}{2} \right) \right) \frac{w_i}{\bm t^2_{ij}} - \bm {tz^R}_{ij}p_{ij} - \bm {tz^I}_{ij}q_{ij} \right) \leq \nonumber \\
& \left( -\bm {t^I}_{ij} - \bm {tz^R}_{ij} \left( \frac{ \bm {b^c}_{ij}}{2} \right) \right) \frac{w_i}{\bm t^2_{ij}} - \bm {tz^R}_{ij}q_{ij} + \bm {tz^I}_{ij}p_{ij} \;\; (i,j) \in E \\
& \tan(\bm {\theta^\Delta}) \left( \left( \bm {t^R}_{ij} - \bm {tz^I}_{ij} \left( \frac{ \bm {b^c}_{ij}}{2} \right) \right) \frac{w_i}{\bm t^2_{ij}} - \bm {tz^R}_{ij}p_{ij} - \bm {tz^I}_{ij}q_{ij} \right) \geq \nonumber \\ 
& \left( -\bm {t^I}_{ij} - \bm {tz^R}_{ij} \left( \frac{ \bm {b^c}_{ij}}{2} \right) \right) \frac{w_i}{\bm t^2_{ij}} - \bm {tz^R}_{ij}q_{ij} + \bm {tz^I}_{ij}p_{ij} \;\; (i,j) \in E 
\end{align}
\end{subequations}
\end{model}

\begin{model}[t]
\caption{The Copper Plate Relaxation of the Extended AC Power Flow  \refcpepf~in Real Numbers.}
\label{model:cp_pf_e_real}
\begin{subequations}
\begin{align}
\mbox{\bf variables: \hspace{-1.3cm} } \nonumber \\
& p^g_i \in (\bm {p^{gl}}_i, \bm {p^{gu}}_i) \;\; \forall i\in N \nonumber \\
& q^g_i \in (\bm {q^{gl}}_i, \bm {q^{gu}}_i) \;\; \forall i\in N \nonumber \\
& w_i \in ((\bm {v^l}_i)^2, (\bm {v^u}_i)^2) \;\; \forall i\in N \nonumber \\
\mbox{\bf subject to: \hspace{-1.3cm} } \nonumber \\
%
%
%
& \sum_{i \in N}  p^g_i - \sum_{i \in N}  {\bm p^d_i} - \sum_{i \in N} \bm g^s_{i} w_i \geq 0 \\ 
& \sum_{i \in N}  q^g_i - \sum_{i \in N} {\bm q^d_i} + \sum_{i \in N} \bm b^s_{i} w_i \geq \sum_{\substack{(i,j)\in E}} -\frac{\bm {b^c}_{ij}}{2} \left( \frac{w_i}{\bm {t}^2_{ij}} +  w_j \right) 
\end{align}
\end{subequations}
\end{model}

\end{document}